\documentclass[12pt]{article}

\usepackage{amssymb}
\usepackage{amsmath}
\usepackage{amsthm}
\usepackage{authblk}
\usepackage{cancel}
\usepackage{caption}
\usepackage{enumerate}
\usepackage{subcaption}
\usepackage{cite}
\usepackage{CJK}
\usepackage{color}
\usepackage{enumerate}
\usepackage{esint}
\usepackage{fancyhdr}
\usepackage{fancyref}
\usepackage{graphics}
\usepackage{graphicx}
\usepackage[colorlinks = true,
            linkcolor = blue,
            urlcolor  = blue,
            citecolor = green,
            anchorcolor = violet]{hyperref}
\graphicspath{ {images/} }
\usepackage{mathtools}
\usepackage{mathrsfs}
\usepackage{pdfsync}
\usepackage{theoremref}
\usepackage{tikz}
\usepackage{xcolor}
\usetikzlibrary{decorations.pathreplacing}

\newtheorem{thm}{Theorem}[section]
\newtheorem{defn}[thm]{Definition}
\newtheorem{rk}[thm]{Remark}
\newtheorem{prop}[thm]{Proposition}
\newtheorem{lem}[thm]{Lemma}

\newtheorem{cor}[thm]{Corollary}

\newcommand{\tA}{\widetilde{A}}
\newcommand{\tE}{\widetilde{E}}
\newcommand{\F}{\mathbf{F}}
\newcommand{\f}{\varphi}
\newcommand{\tF}{\widetilde{\mathbf{F}}}

\newcommand{\tH}{\widetilde{H}}

\newcommand{\tM}{\widetilde{M}}
\newcommand{\N}{\mathbb{N}}
\newcommand{\R}{\mathbb{R}}
\renewcommand{\S}{\Sigma}
\renewcommand{\SS}{\mathbb{S}}
\newcommand{\p}{\pi}

\newcommand{\n}{\nu}

\newcommand{\W}{\Omega}

\newcommand{\e}{\varepsilon}
\renewcommand{\l}{\lambda}

\newcommand{\tn}{\tilde{\nu}}
\newcommand{\tr}{\widetilde{\rho}}

\newcommand{\bd}{\partial}
\renewcommand{\c}{\nabla}
\newcommand{\tc}{\widetilde{\nabla}}
\newcommand{\dtm}{\: d \widetilde{\mu}}
\newcommand{\sub}{\subset}
\newcommand{\wo}{\setminus}

\newcommand{\dsp}{\displaystyle}

\title{\vspace{-1 in}A Liouville Theorem for Mean Curvature Flow}
\author{Kevin Sonnanburg \\ email \href{mailto:sonnanburg@math.utk.edu}{sonnanburg@math.utk.edu}}
\affil{University of Tennessee}

\begin{document}

\maketitle 

\begin{abstract}
	Ancient solutions arise in the study of parabolic blow-ups. 
	If we can categorize ancient solutions, we can better understand blow-up limits. 
	Based on an argument of Giga and Kohn in~\cite{gk}, we give a Liouville-type theorem restricting ancient, type-I, non-collapsing two-dimensional mean curvature flows to either spheres or cylinders. 
\end{abstract}

%\textcolor{blue}{can I get rid of: \\
%singular point? }

\setcounter{section}{-1}

\section{Introduction}

 We study ancient solutions to mean curvature flow. 
Let $\F:\mathcal{M}\times \R^- \to \R^{N+1}$ be a family of smooth embeddings $\mathbf{F}(\cdot, t) = M(t)$, where $\mathcal{M}$ is a closed $N$-dimensional manifold. 
	We say that $M=\left\{ M(t) \right\}_{t \in [0,T)}$ is a mean curvature flow if
		\begin{equation} \label{eqn:mcf}
		\bd_t \F = - H \nu ,
	\end{equation}
where $H$ is the scalar mean curvature, $\nu$ is the \emph{outward} unit normal, and $-H \nu$ is the mean curvature vector. \\

We call a mean curvature flow \emph{ancient} if it is defined for all negative time. 
Ancient solutions arise as blow-ups of singularities (see the discussion after \thref{defn:r} for rescaling below for one way this can be done). 
Daskalopoulos, Hamilton, and \v{S}e\v{s}um completely classified ancient convex solutions for embedded curves in $\R^2$ in~\cite{dhs}. 
Here our goal is to further the classification to two dimensions for mean-convex, type-I, non-collapsed flows. 
At any point in time, an ancient solution has had an arbitrarily long amount of time for diffusion to take place, so we expect it to be highly regular and symmetric. 
We see this in the work of Huisken and Sinestrari in~\cite{hs} where they show, assuming convexity and compactness, a number of conditions equivalent to the flow being a shrinking spere. 
This is similar to our result here, so we emphasize that although we impose other restrictions, we allow for compactness or noncompactness. 
(Haslhofer and Kleiner show in~\cite{hk} that ancient mean-convex, non-collapsing solutions are convex anyway.)
\\

In the theorem, we do assume some regularity to begin with. 
In the spotlight are the type-I curvature bound and the non-collapsing condition. 
With the type-I assumption, we show that an eternal solution for the rescaled flow, as in \thref{defn:r} (see~\cite{hui}), all orders of curvature are bounded in time. 
The non-collapsing condition prevents sheeting, thereby preserving embeddings as $t \to -\infty$. 
This is important for integral convergence if one intends to integrate on the embedded hypersurface itself, rather than a background manifold. 
Both assumptions are rather strong, but since we have in mind ancient solutions which arise from blow-ups at singularities of type-I, mean-convex, compact flows, both are quite reasonable. 
\\

There are examples of ancient solutions that do not satisfy the conlusions of our main theorem. 
The paperclip solution, one of the two classes in~\cite{dhs}, converges to two parallel lines as $t \to -\infty$, but behaves like the grim reaper solution at either end. 
This was generalized in a sense by White in~\cite{whimc} to higher dimensions, but was studied in more detail by Haslhofer and Hershkovitz in~\cite{hh}. 
The paperclip, however, is neither type-I, nor non-collapsing, as $t \to -\infty$. 
\\

The method here is inspired by that of Giga and Kohn in~\cite{gk}. 
There they show that the rescaled limits as $t \to -\infty$ and $t \to +\infty$ are the same. 
They then classify self-similar solutions to find that the forward and backward limits of the rescaled solution must have the same energy. 
The energy they use is decreasing, so once they relate it to the time derivative of the solution, they can integrate across time to show the the solution is constant in time. 
\\

We can build off the work of Huisken in~\cite{hui} or White in~\cite{whimc} to classify the forward limit, and the work of Haslhofer and Kleiner in~\cite{hk} to classify the backward limit. 
However, the geometric nature of the flow adds a complication: there are different self-similar solutions that can arise as blow-ups and blow-downs, and they have different energies. 
We calculate the energy (Huisken's Gaussian area functional defined in~\cite{hui}) explicitly in each case. 
The fact that energy is decreasing means that the backward limit cannot have a lower engergy than the forward limit, but this does not cover the case when the backward limit has a strictly higher energy than that of the forward limit. 
We see in the proof of \thref{prop:limit} that the only case in which the monotonicity does not help is a noncompact backward limit with a compact forward limit. 
This case is ruled out rather directly in \thref{lem:evo}, since the rescaled evolution equation tends to expand the hypersurface. 
\\

We now give some definitions so we can state the main theorem (\thref{thm:main}). 

%\begin{defn}[Properly Embedded]
%	\thlabel{defn:pe}
%	Let $\M$ be an $N$-dimensional manifold and $\F:\M \to \R^{N+1}$ be a parameterization of a hypersurface $\Sigma \subset \R^{N+1}$. 
%	We say $\Sigma$ is \emph{properly embedded} if, for every compact subset $K$ of $\R^{N+1}$, $\F^{-1}(K \cap \Sigma)$ is compact in $\M$. 
%\end{defn}

\begin{defn}[Singular Point]
	\thlabel{defn:sing}
	We say $x \in \R^{N+1}$ is a singular point if there is a sequence $(p_i,t_i) \in \mathcal{M} \times \R^-$ with $t_i \nearrow 0$ such that $\F(p_i,t_i) \to x$ and $|A(p_i,t_i)| \to \infty$ as $i \to \infty$. 

		All mentions of singularities are at the first singular time $t=0$. 
	\end{defn}

\begin{defn}[Type-I Flow]
	Let $M$ be a mean curvature flow for times $t \in \R^-$. 
	Let $\l(t) = (-2t)^{-\frac{1}{2}}$, and write $A$ for the second fundamental form. 
	We say $M$ is type-I if there is a $C_0>0$ so that 
	\[
		\max_{x \in M(t)} |A(x,t)| \le C_0 \l(t) \mbox{ for } t \in \R^- .
	\]
\end{defn}

\begin{rk}
	The type-I condition is typically employed in discussions of blow-ups at singularities. 
	However we apply the condition to the entirety of an ancient flow, meaning curvature decays as $t \searrow -\infty$ as well. 
\end{rk}

\begin{defn}[Polynomial Volume Growth]
	We say a surface $\Sigma \in \R^3$ has \emph{polynomial volume growth} if $Vol(B_R(0) \cap \Sigma)$ is bounded by some polynomial $P(R)$. 
	(Volume here refers to the intrinsic volume, in this case area.)
	
	We say a mean curvature flow $M$ has \emph{uniform polynomial volume growth} if, for every $t$ that $M$ is defined, $M(t)$ has polynomial volume growth, where the polynomial $P(R)$ is independent of $t$. 
\end{defn}

\begin{defn}[Non-Collapsing Condition]
	From Definition 1 of~\cite{and}: We say a mean-convex hypersurface $M_0$ bounding an open region $\W$ in $\R^{N+1}$ is $\alpha$-non-collapsed if, for every $x \in M_0$, there exists a sphere of radius $\frac{\alpha}{H(x)}$ contained in $Cl(\W)$, and another contained in $\W^c$, tangent to $M_0$ at $x$. (See Figure~\ref{fig:andrews}). 
\\

\begin{figure}[h]
	\centering
\includegraphics[scale=.2]{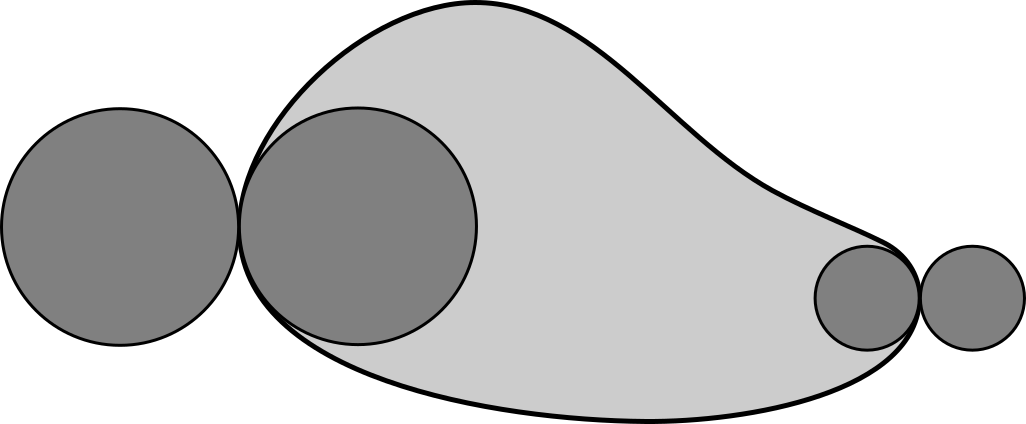}
\caption{Spheres of radii varying with curvature.}
\label{fig:andrews}
\end{figure}
\end{defn}

\begin{thm}[Main Theorem]
	\thlabel{thm:main}
	Let $M(t)$ be a smooth, properly embedded, complete, ancient, type-I, mean-convex, $\alpha$-non-collapsed, two-dimensional mean curvature flow in $\R^3$ with first singular point $x$ at time $t=0$.
	Further assume that $M(t)$ has uniform polynomial volume growth on $\R^-$. 

	Then $M(t)$ is either a sphere or cylinder, shrinking homothetically until it vanishes at time $t=0$. 
\end{thm}

\begin{rk}
	The assumption that $N=2$ is necessary to restrict the topologies of blow-ups at a singular point. 
	See \thref{prop:limit} and the discussion before it for further explanation. 
\end{rk}

\begin{rk}
	Of course, for manifolds without boundary, properly embedded implies completeness. 
	Furthermore, we have by Corollary 1.6 of~\cite{sonc}, that $\alpha$-non-collapsed implies properly embedded and uniform polynomial volume growth. 
\end{rk}

We also need the following definitions for the proof. 

\begin{defn}(Gaussian Area)
	\thlabel{defn:r}
	For a flow $M(t)$ of surfaces in $\R^3$, define 
	\[
		E_{(x_0,t_0)}(t) 
		= \int_{M(t)} \rho_{(x_0,t_0)}(x,t) \: d \mu ,
	\]
	where 
	\[
		\rho_{(x_0,t_0)}(x,t) = \frac{1}{\left( 4\pi(t_0- t) \right)^{\frac{N}{2}}} 
		e^{\frac{-|x-x_0|^2}{4(t_0-t)}} .
	\]
\end{defn}
	We will mostly be assuming $(x_0,t_0)=(0,0)$. 
	In that case, we \emph{omit} the subscript. 
	That is, $E:=E_{(0,0)}$ and $\rho:=\rho_{(0,0)}$. 

\begin{defn}[Rescaled Flow]
	Let $\F$ be a parameterization of a mean curvature flow. 
	Let $\lambda(t) = (2(-t))^{-\frac{1}{2}}$, $\xi = \l(t) x$, and $s = -\frac{1}{2} \log(-t)$. 
	Without loss of generality, we assume throughout the paper that the singular point is 0. 
	So define the rescaled flow
	\[
		\tF(p,s) := \lambda(t) \F(t) .
	\]
	Further define 
	\[
		\tE(s) = 
		\int_{\tM(s)} \tr(x) \dtm(s) ,
	\]
	where 
		\[
			\tr(x) = e^{\frac{-|x|^2}{2}} .
		\]
\end{defn}
	As introduced in~\cite{hui}, the new flow satisfies the equation 
	\begin{equation}
		\label{eqn:r}
		\partial_s \tF = \tF-\tH \nu ,
	\end{equation}
	where $\tH = \tH\left( \tF(p,s) \right)$ 
		and $\tn = \tn\left( \tF(p,s) \right)$. 
	We will assume for simplicity that the singular point in question is the origin, so we only need to rescale around 0. 
		Notice that $e^s = \sqrt{2} \l(t)$.
		\\

	With this definition in mind, we can see one way to arrive at an ancient solution in the study of blow-ups. 
	For a non-ancient solution of mean curvature flow defined for times starting at $t=0$ and first singularity at $(x,t)=(0,T)$, the solution rescaled around the singularity is defined for $s \in [s_0,\infty)$, where $s_0 = -\frac{1}{2} \log T$. 
	If we define $\tM_n(s)=\tM(s+s_n)$ with $s_n \nearrow \infty$, a limit solution is obtained given enough curvature control (as one would have if the flow is type-I). 
		
		\begin{rk}
			If $M$ is type-I, then $\tH$ uniformly bounded for all time, even if $M$ is ancient. 
		\end{rk}

%		\begin{defn}[Tangent Flow]
%			Let $M$ be a mean curvature flow for times $t \in \R^-$.
%			Fix $(x_0,t_0) \in \R^{N+1} \times \R^+$.
%			Then one can check that the rescaling
%			\[
%				M_{\mu,(x_0,t_0)}(t) := \mu \left( M\left(t_0- \mu^{-2} (-t) \right) -x_0 \right) \mbox{  for  } t \in \left[-\infty, -\mu^2 t_0 \right)
%				\]
%				is also a solution to mean curvature flow.
%				Taking a sequence $\mu_i \nearrow \infty$, consider the sequence of rescalings $M^i_{(x_0,t_0)}(t):= M_{\mu_i,(x_0,t_0)}(t)$.
%				If $M_{(x_0,t_0)}^i(t)$ has a subsequence converging weakly in $i$ to a flow $M^{\infty}(t)$ (in the sense of measures), that limit is called a \emph{tangent flow}.
%			\end{defn}

\begin{defn}[Local Graph Convergence]
	\thlabel{defn:to}
	Assume $k \ge 1$. 
	Let $\S$ and $\S_n$ be $k$-smooth, properly embedded hypersurfaces in $\R^{N+1}$. 
	Assume $\S$ is oriented by a smooth normal vector field $\n$. 
	We say $\S_n$ converges to $\S$ locally in the graph sense to order $k$  if the following holds:

	For every open ball $B \sub \R^{N+1}$, there is $n_0 >0$ so that whenever $n \ge n_0$
	\begin{enumerate}[i)]
		\item The limit set $\S$ is the set of all accumulation points of $\S_n$. 
			That is $\S$ is the set of all $x \in \R^{N+1}$ such that there is a sequence of points $x_n \in M_n$ with $x_n \to x$. 
		\item If $\S \cap B$ is nonempty, the nearest point map 
			\[
				\p_n^B: \S_n \cap B \to \S 
			\]
			is a well-defined diffeomorphism onto its image $V_n^B \sub \S$. 
		\item For $y \in M_n \cap B$, write $x=\pi_n^B(y)$. 
			Then define $g_n^B: V_n^B \to \R$  to be the height function
			\[
				g_n^B(x)=(y-x) \cdot \n(x) 
			\]
			over $V_n^B \sub \S$ so that
			\[
				(\p_n^B)^{-1}(x) = x + g_n^B(x) \n(x) 
			\]
			so that $g_n^B$ is the signed height of $V_n^B \sub \S_n$ over $\S \cap B$. 
			Then for every $k \in \N$, 
			\[
				\|g_n^B\|_{C^k(V_n^B)} \xrightarrow[n \to \infty]{} 0. 
			\]
	\end{enumerate}

\end{defn}

\section{Some Technical Lemmas}
  
\subsection{mean curvature flow background}

\begin{lem}[Proposition 2.3 of \cite{hui}]
	\thlabel{lem:huibound}
	Given $s_0 \in \R$ (and corresponding $t_0$), for each $m > 0$, there is $C(m) < \infty$, such that $|\tc^m \tA|^2 < C(m)$ holds on $\tM(s)$ uniformly in $s$, where $C(m)$ depends on $N$, $m$, $C_0$, and $M(t_0)$. 
\end{lem}

This phrasing is changed slightly to accomodate ancient solutions by choosing $M(t_0)$ as ``initial data''.  

\begin{lem}[Corollary 3.2 of \cite{hui}]
	\thlabel{lem:dE}
	For the rescaled flow $\tM$, 
	\[
		\bd_s \tE(s) = \int_{\tM(s)} \left|\tF^{\perp}-\tH \tn\right|^2 \tr \dtm . 
	\]
\end{lem}

\begin{lem}
	\thlabel{lem:Es}
	For a mean curvature flow $M$ and rescaled flow $\tM$, $\tE(s) = (2\pi)^{\frac{N}{2}} E(t)$. 
\end{lem}

The proof is a direct calculation.

\subsection{Some Calculus}

\begin{lem}[Backwards Compactness Preservation]
	\thlabel{lem:evo}
	If $M$ is a compact, type-I mean curvature flow, then $\tM(s)$ is uniformly  bounded for all times $s \le s_0$. 
\end{lem}

\begin{proof}
%	We'll need Hessiany things, so we have
%	\begin{align*}
%		D|\tF|^2 &= 2 \tF \cdot D \tF 
%		= 2 \tF \cdot I
%		= 2 \tF \\
%%
%		D^2|\tF|^2 &= 
%			D(2\tF) 
%			= 2I \\
%%
%			D^2_{\nu,\nu} |\tF|^2 
%			&= \nu \cdot 2 I \cdot \nu 
%			= 2 \\
%%
%		\Delta_{\R^{N+1}} |\tF|^2 &= \trace(2I) = 2 N
%	\end{align*}
%	
%	Also, by (A.11) of~\cite{eck},
%	\begin{align*}
%		\tD |\tF|^2 
%		&= \Delta_{\R^{N+1}} |\tF|^2 - D_{\nu,\nu}^2 |\tF|^2 + \tbH \cdot D|\tF|^2 \\
%		&= 2N-2+2\tF \cdot \tbH ,
%	\end{align*}
%	or
%	\[
%		-2\tbH \cdot \tF 
%		= 2(N-1) - \tD |\tF|^2 .
%	\]

	From the type-I bound, we know $|\tH| \le C_0$. 
	Going back in time, 
	\begin{align*}
		\bd_{-s} |\tF|^2
		&= - \bd_{s} |\tF|^2
		= -2 \tF \cdot \bd_s \tF 
		= -2 \tF \cdot (\tF - \tH \tn) \\ 
		&= -2 \tF \cdot \tF + 2 \tF \cdot \tH \tn 
		\le -2|\tF|^2 + 2 C_0 |\tF| .
	\end{align*}

	So $\bd_s |\tF|^2$ is strictly negative whenever $|\tF| > C_0$. 
	Let 
	\[
		\Lambda:= \max\left\{ C_0, \max_{\tM(s_0)} |\tF| \right\} . 
	\]
	Thus, going back in time, $\tM(s)$ cannot escape the ball $B_{2\Lambda}(0)$. 
\end{proof}

\begin{cor}
	\thlabel{cor:ball}
	Assume $M$ has a singular point at 0. 
	Then $\tM(s) \cap B_N(0)$ is nonempty for ever $s \in \R$. 
\end{cor}

%In particular, due to \thref{lem:limit}, \thref{cor:ball} applies to rescaled flows centered at singular points. 

\begin{proof}
	Assume without loss of generality that 0 is a singular point. 
	By Proposition 2.2.6~\cite{mant}, $M(t) \cap Cl(B_{\sqrt{-2Nt}}(0))$ is nonempty for all time. 
	Rescaling by $\l(t)=(-2t)^{-\frac{1}{2}}$, we find that $\tM(s) \cap Cl(B_{\sqrt{N}})$. 
	Our conclusion immediately follows. 

%	By \thref{lem:TI}, 
%	\[
%		|\tF^{p^*}(p,s)|
%		= |\l(t) \left( \F(p,t)-p^* \right)|
%		\le \l(t) C_0 \l^{-1}(t) 
%		= C_0 
%	\]
%	for all $s \in \R$. 
%	Then $\tF^{p^*}(p,s) \in B_{2C_0}(0)$ for all $s \in \R$. 
%	Therefore $\M(s)$ intersects $B_{2C_0}(0)$ for every $s \in \R$. 
\end{proof}

%Note the proof is with regards to the $p$ corresponding to $p^*$, not $|\tF|$ globally. 

	\section{Regularity}

We will need two time derivatives of $E$ later, which involves fourth order terms, so we need high-regularity control to properly manage convergence. 
Huisken takes care of this forward in time in~\cite{hui}, but the proof relies on a maximum principle. 
We need to prove bounds for $|\tc^m \tA|$ backward as well.
We refer to a parabolic regularity result in~\cite{eck}.

\begin{lem}[Proposition 3.22 of~\cite{eck}]
	\thlabel{lem:eck}
	Let $(M_t)$ be a smooth, properly embedded solution of mean curvature flow in $B_{\rho}(x_0) \times (t_0-\rho^2,t_0)$
	\[
		|A(x)|^2 \le \frac{c_0}{\rho^2}
	\]
	for all $t \in (t_0-\rho^2,t_0)$ and $x \in M_t \cap B_{\rho}(x_0)$. 
	Then for every $m \in \N$ there is a constant $c_m=c_m(N,m,c_0)$ such that for all $x \in M_t \cap B_{\frac{\rho}{2}(x_0)}$ and $t \in (t_0-\frac{\rho^2}{4})$,
	\[
		|\c^m A(x)|^2 \le \frac{c_m}{\rho^{2(m+1)}} .
	\]
\end{lem}
	\begin{center}
		\begin{tikzpicture}
			\draw[fill] (0,0) circle (.05);

			\draw[ultra thick,blue] (0,0) ellipse (1 and .3);
			\draw[ultra thick,blue] (-1,0) -- (-1,-1);
			\draw[ultra thick,blue] (1,0) -- (1,-1);
			\draw[ultra thick,blue] (-1,-1) arc (180:360:1 and .3);
			\draw[ultra thick,blue,dashed] (1,-1) arc (0:180:1 and .3);

			\draw[ultra thick] (0,0) ellipse (2 and .6);
			\draw[ultra thick] (-2,0) -- (-2,-4);
			\draw[ultra thick] (2,0) -- (2,-4);
			\draw[ultra thick] (-2,-4) arc (180:360:2 and .6);
			\draw[ultra thick,dashed] (2,-4) arc (0:180:2 and .6);

			\node[right] at (0,0) {$x_0$};
			\node[right] at (2,0) {$t=t_0$};
			\node[right] at (2,-4) {$t=t_0-\rho^2$};
			\draw[blue] (1.2,-1) -- (2.5,-1) node[right] {$t=t_0-\frac{\rho^2}{4}$};
			\draw[decoration={brace,raise=25pt},decorate,ultra thick] (0,0) -- node[above=30pt] {$\rho$} (2,0);
			\draw[decoration={brace,mirror,raise=15},decorate,ultra thick,blue] (0,-1) -- node[below=20pt] {$\frac{\rho}{2}$} (1,-1);

			\draw[blue] (-.8,-.7) -- (-2.5,-.7) node[left] {$|\c^m A|^2 \le \frac{c_m}{\rho^{2(m+1)}}$};
			\draw (-1.3,-2.5) -- (-2.5,-2.5) node[left] {$|A|^2 \le \frac{c_0}{\rho^2}$};
		\end{tikzpicture}
	\end{center}

%However, we want something applicable to the rescaled $|\tA|$. 
%Looking carefully at the proof of Lemma \ref{lem:eck}, it is based on a maximum principle, wherein the key calculation gives in 
%	\[
%		\left(\partial_t - \Delta \right)f \le -\delta f^2 +K 
%	\]
%where $\d,K$ are constants, and 
%	\[
%		f = |\c A|^2 (\L_0 + |A|^2)
%	\]
%with $\Lambda_0$ depending only on $c_0$. \\
%
%Well, the only substantive change in the proof is the extra terms in $\bd_s \tilde{f}$ coming from the rescaling. 
%However, $|\c A|$ and $|A|$ are both of nonpositive degree (see~\cite{hui}), so the extra terms coming from them can be dropped from the inequality. 
%The rest of the proof applies, yielding constant bounds on all derivatives $|\tc^m \tA|$ backwards in time, as well as fowards (see~\cite{hui}).
%The final result is accomplished by shifting the cylinder $B_{\rho}(x_0)\times (s_0-1,s_0)$ around in space-time. \\

%\textcolor{red}{If that doesn't work, I can get the equivalence of TI and a $|\D H|$ bound, then all that is left in $\bd_s$ of the integrand in $E$ is $|\c H|$}.

	Now we want to use the above lemma to get a bound on covariant derivatives in the rescaled flow, and we want to do so for all time. 
	For $s>0$, Huisken did this in Proposition 2.3 of~\cite{hui} (\thref{lem:huibound} of this work). 
	For $s<0$, we take advantage of the type-I bound. 
	In the nonrescaled setting, going farther back in time forces the curvature to decay. 
	This allows us to choose larger $\rho$ for more control on $|\c A|$. 

	\begin{lem}[Ancient Regularity]
	\thlabel{lem:ab}
	Let $M$ be a smooth, properly embedded, ancient, type-I mean curvature flow.
	Then for $m \in \N$, there is $c_m>0$ so that for each $t \in \R^-$,
	\[
		|\c^m A| \le \sqrt{c_m} \l^{m+1}(t) 
		= \sqrt{c_m} (-2t)^{-\frac{m+1}{2}} 
	\]
	uniformly over $M(t)$. 
\end{lem}

	\begin{proof}
	Let $t \in \R^-$, and $x \in M(t)$. 
	With the goal of applying \thref{lem:eck}, choose $\rho=\l^{-1}(t)=(-2t)^{-\frac{1}{2}}$, $x_0=x$, and $t_0 = t+\frac{\rho^2}{8}=\frac{3}{4}t < 0$. 
	That puts our point of interest, $(x_0,t)=(x,t)$, at the center of the inner cylinder, with $t_0$ at the top of the cylinders. 
	\\

	Due to the type-I bound, $|A(y,\tau)| \le C_0 \l(t_0)$ for every $(y,\tau)$ in the outer cylinder since $\tau \le t_0$. 
	Now setting $c_0 = \frac{2}{\sqrt{3}}C_0$,
		\[
			|A(y,\tau)|
			\le C_0 \l ( t_0 ) 
			= C_0 \l\left( \frac{3}{4}t \right) 
			= C_0 \frac{2}{\sqrt{3}} \l(t) 
			= \frac{c_0}{\rho} .
		\]
		Now recall $(x,t)$ is in the inner cylinder. 
		Then since $|A| \le \frac{c_0}{\rho^2}$ in the outer cylinder, \thref{lem:eck} says that for every $m \in \N$, there is $c_m$ so
		$|\c^m A|^2 \le \frac{c_m}{\rho^{2(m+1)}}$ in the inner cylinder. 
		Rather,
		\[
			|\c^m A(x,t)|
			\le \frac{\sqrt{c_m}}{\rho^{m+1}} 
			= \sqrt{c_m} \l^{m+1}(t) .
		\]

\end{proof}

\begin{cor}[Eternal Regularity]
	\thlabel{cor:bound}
	Let $M$ be a smooth, properly embedded, ancient, type-I mean curvature flow. 
	Then for $m \in \N$, there is $C_m>0$ so that 
	\[
		\sup_{\xi \in \tM(s), s \in \R} |\tc^m \tA| \le C_m .
	\]
\end{cor}

\begin{proof}
	Recall $\l(t) = \frac{e^s}{\sqrt{2}}$, so that we have from \thref{lem:ab}, 
	\[
		|\tc^m \tA|
		= \l^{1-m} |\c^m A|
		\le \sqrt{c_m} \l^{1-m} \l^{m+1} 
		= \sqrt{c_m} \l^2 
		= \frac{\sqrt{c_m}}{2} e^{2s}.
	\]

	Now, for $s \in \left( -\infty,0 \right)$ , $|\tc^m \tA| \le \frac{\sqrt{c_m}}{2}$. 
	Then \thref{lem:huibound} provides a $C_m \ge \frac{\sqrt{c_m}}{2}$ for which $|\tc^m \tA| \le C_m$ for $s \in \left( 0, \infty \right)  $. 
	Therefore, 
	\[
		|\tc^{m} \tA| \le C_m
	\]
	for all time. 
\end{proof}

\section{Proving the Main Theorem}

\begin{thm}[Subsequential Limits]
		\thlabel{thm:subseq}
		Let $M$ be a smooth, properly embedded, ancient, type-I, mean-convex, two-dimensional mean curvature flow with uniform polynomial growth. 
		Assume $M$ has a singular point at the origin at time $t=0$. 

		Then for every sequence of rescaled times $s_i \searrow \infty$, there is a subsequence $\left\{ s_{i_j} \right\}$ so that $\dsp \lim_{j \to \infty} \tM(s_{i_j})$ converges to some $\tM_{-\infty}$ in $C_{loc}^2$ in the graph sense. 
		Furthermore, $\tM_{-\infty}$ is either a plane passing through 0, a cylinder centered at 0 with radius 1, or a sphere centered at 0 with radius $\sqrt{2}$. 

		All the same can be said of some sequence $s_i \nearrow \infty$ and a limit $\tM_{+\infty}$. Although in that case we can rule out the plane. 
	\end{thm}

	\begin{proof}
		Since $|\c^m A| \le C_m$ by \thref{cor:bound} and $\tM(s) \cap B_{N}(0)$ is nonempty by \thref{cor:ball}, $\tM_{-\infty}$ exists by Corollary 1.6 of~\cite{sonc}. 
		We know from (5) of~\cite{son} that $\tM_{-\infty}$ is a tangent flow, or blowdown soliton. 
		Therefore Theorem 1.11 of~\cite{hk} says that $\tM_{-\infty}$ is either a plane, cylinder, or sphere. 
		%\textcolor{red}{what is $C_{loc}^2$ on $\R^N$?}
		\\

		Again, $\tM_{+\infty}$ exists due to Corollary 1.6 of~\cite{sonc}. 
		Since $\tM_{+\infty}$ is a tangent flow, we know it is either a plane, cylinder, or sphere by Theorem 1 of~\cite{whimc}. 
		However, Corollary 1.8 of~\cite{sw} rules out the plane for tangent flows at first singularities for mean-convex flows. 
		\\

		It follows from (\ref{eqn:r}) that for a stationary sphere or cylinder, $\tH = \tF \cdot \tn$. 
		The necessary radii follow directly from there.

	\end{proof}

\begin{lem}
		The limits $\dsp \tE_{\pm \infty} := \lim_{s \to \pm \infty} \tE(s)$ exist. 
	Furthermore, the limits $\tE_{\pm\infty}$ are equal to the Gaussian areas of $\tM_{\pm\infty}$. 
\end{lem}

\begin{proof} ~

	\paragraph{The Limits $\tE_{\pm \infty}$ Exist}
		Since $M(t)$ exhibits uniform polynomial volume growth, $E(t)$ is bounded for $t \in \R^-$. 
		Then by \thref{lem:Es}, $\tE(s)$ is also bounded for all $s \in \R$. 
		We know from \thref{lem:dE} that $\tE$ is decreasing in time and bounded below by 0. 
		Therefore, its limits at times $\pm \infty$ both exist. 
		We denote them $\tE_{\pm \infty}$. 

		\paragraph{Gaussian areas}
		We do the proof for $\tM_{-\infty}$, and the proof for $\tM_{+\infty}$ is identical. 
		One will notice below that different radii $R+\e$ and $R$ are used in the domains for integrals. 
		This is of little interest, but necessary to accomodate the normal vectors to $\tM_{-\infty} \cap B_R(0)$, which leave the ball near the boundary. 
		\\

Let $0 < \e < 1$. 
By uniform polynomial volume growth, there exists $R >0$ such that 
	\[
		\int_{\tM(s_i) \wo B_R(0)} \tr \dtm_i < \e
	\]
	for all $i$ and also for $\tM(s_i)$ replaced by $\tM_{-\infty}$. 
	By Corollary 1.6 of~\cite{sonc}, for large $i$ there are open $V_i \sub \tM_{-\infty} \cap B_{R+\e}(0)$ and $f_i:V_i \to \R$ with $\|f_i\|_{C^1} < \e$ such that 
	\[
		\f_i(x) := x + f_i(x) \tn_{-\infty}(x) 
	\]
	is a diffeomorphism from $V_i$ onto $\tM(s_i) \wo B_R(0)$. 
	\\

	Then 
	\[
		\int_{\tM(s_i) \cap B_R(0)} \tr \dtm_i
		= \int_{\tM_{-\infty} \cap B_{R+\e}(0)} \chi_{V_i} \tr(\f_i(x)) \sqrt{1+|\tc_{-\infty} f_i|^2} \dtm_{-\infty} ,
	\]
	where the integrals now have a fixed domain, and $\chi_{V_i}$ is the characteristic function 
	The integrand is bounded by 2 and converges pointwise to $\tr$. 
	Therefore we can apply dominated convergence. 
	Taking $i$ large enough, and repeatedly absorbing $O(\e)$-terms, we write
	\begin{align*}
		\int_{\tM(s_i)}\tr \dtm_i 
	&	= \int_{\tM(s_i) \cap B_R(0)} \tr \dtm_i + O(\e) \\
	&= \int_{\tM_{-\infty} \cap B_{R+\e}(0)} \tr \dtm_{-\infty} +O(\e) \\
	& = \int_{\tM_{-\infty}} \tr \dtm_{-\infty} + O(\e) , 
	\end{align*}
	where we used 
	\[
		\int_{\tM_{-\infty} \wo B_{R+\e}(0)} \tr \dtm_{-\infty} 
		\le \int_{\tM_{-\infty} \wo B_R(0)} \tr \dtm_{-\infty} .
	\]

%	We focus on $\tE_{-\infty}$, but the proof for $\tE_{+\infty}$ is identical. 
%	\\
%
%		This is a matter of applying dominated convergence, so given a compact subset $K \subset \R^{N+1}$, we must bound $\tr \dtm$ on $K$ uniformly in $s$. 
%		Obviously $\tr$ is bounded, so we need only control the area element. 
%		This is probably not possible when considering only the Jacobian of the map $\tF$. 
%		However, given the graph convergence of the afforementioned subsequence, we can instead consider the area elements of the graphs. 
%		\\
%
%		In each compact subset, we now know 
%		\[
%			\tM(s_{i_j}) \xrightarrow[s_{i_j} \searrow -\infty]{C^2} \tM_{-\infty},
%		\]
%		so around each point of $\tM_{-\infty}$ is a small neighborhood in which we can think of $\tM(s_{i_j})$ as a graph over $\tM_{-\infty}$, when $j$ is large, by \thref{lem:conv}. 
%		Suppose, in the compact set $K$, $\tM(s_{i_j})$ is expressed as a graph of the function $f_{i_j}$ over $\tM_{-\infty}$. 
%		Then the area element of $\tM(s_{i_j})$ is $\sqrt{1+|\c f_{i_j}|^2}$. 
%		Given the convergence, we have $\|f_{i_j}\|_{C^1} \to 0$, so the area element is bounded in $K$ uniformly in $s$. 
%		\\
%
%		Now, by dominated convergence, on any compact subset of $\tM_{\infty}$ can be covered by finitely many such small neighborhoods, so we have for any compact $K \subset \R^{N+1}$,
%		\[
%			\int_{\tM(s_{i_j}) \cap K} \tr \dtm 
%			\to \int_{\tM_{-\infty} \cap K} \tr \dtm
%		\]
\end{proof}

%	Now there is the question of uniqueness of $\tM_{-\infty}$ with respect to the subsequence. 
%	If we can restrict the possible geometries of $\tM_{-\infty}$, then we will have an easier time comparing $\tM_{-\infty}$ and $\tM_{+\infty}$. 
%	As in~\cite{hui}, we use the convergence of $\tE$ to $\tE_{-\infty}$. 
%
%
%	Keep in mind in the next theorem that $\tM_{-\infty}$ may depend on subsequence, but $\tM_{\infty}$ does not. 
%\\

The following result is where we really need $N=2$. 
That is, if $\tM_{-\infty}$ and $\tM_{+\infty}$ can be generalized cylinders, our method does not prevent them from being generlized cylinders with different numbers of flat factors. 
\thref{lem:evo} lets us handle the case where either limit is a (compact) sphere, and we are able to rule out planes altogether. 
Restricting our scope to surfaces means the only other possiblity is cylinders with the known factorization $\SS^1 \times \R^1$.

\begin{prop}[$\tM_{-\infty} \cong \tM_{\infty}$]
		\thlabel{prop:limit}
		Let $M$ be a smooth, complete, properly embedded, ancient, type-I, mean-convex, two-dimensional mean curvature flow with uniform polynomial volume growth. 
		Assume $M$ has a singular point at the origin at time $t=0$.

		Then $\tM_{-\infty}$ and $\tM_{\infty}$ are either both spheres or are both cylinders. They have the same radius, and are centered at the origin. 
	\end{prop}

	\begin{proof}
		Recall from \thref{thm:subseq} we know that $\tM_{-\infty}$ is either a plane, cylinder, or sphere, and $\tM_{+\infty}$ is only a cylinder or sphere. 
		\\
		
		Now we turn our attention to determining possibile shapes for $\tM_{-\infty}$. 
		The strategy is to use the monotonicity of $\tE$ to rule out the plane, then show that $\tM_{-\infty}$ if and only if $\tM_{\infty}$. 
		If $\tE_P$, $\tE_C$, and $\tE_{S}$ are the Gaussian areas for the plane, cylinder of radius 1, and sphere of radius $\sqrt{2}$ respectively, a direct calculation gives
		$\tE_P=2\pi$, $\tE_C=2\pi \sqrt{\frac{2\pi}{e}}$, and $\tE_S=2\pi \frac{4}{e}$. 
		That is
		\[
			\tE_P < \tE_S < \tE_C .
		\]
		
		First suppose $\tM_{-\infty}$ is a plane. 
		We already know $\tM_{\infty}$ is a cylinder of radius 1 or a sphere of radius $\sqrt{2}$. 
		However that would mean $\tE$ increased, which is a contradiction. 
		\\

		If either $\tM_{-\infty}$ or $\tM_{\infty}$ is a sphere, then there is $s \in \R$ so that $\tM(s)$ is compact. 
		Thus by \thref{lem:evo}, $\tM$ is a compact flow. 
		Therefore both $\tM_{-\infty}$ and $\tM_{\infty}$ must be the same sphere. 
		\\

		Now we have that $\tM_{+\infty}$ is a sphere if and only if $\tM_{-\infty}$ is a sphere. 
		Then, by process of elimination, $\tM_{+\infty}$ is a cylinder if and only if $\tM_{-\infty}$ is a cylinder. 
		Thus $\tM_{+\infty}$ must be isometric to $\tM_{-\infty}$, since \thref{thm:subseq} ensures they have the same radius. 
		Due to the equations $\tF_{\pm \infty} \cdot \tn_{\pm \infty} = \tH_{\pm \infty}$, the sphere or cylinder must be centered around the origin.
	\end{proof}

%Since we generally know more about $\tM_{+\infty}$, we will break our investigation into the cases where $\tM_{+\infty}$ is a sphere or a cylinder. 
%
%\subsection{Spherical case}
%Of spheres, cylinders, and planes, spheres have the highest Gaussian area. 
%The Gaussian area is decreasing in time, so we must have $E_{-\infty} \ge E_{\infty}$. 
%Therefore, if $\tM_{\infty}$ is a sphere, so is $\tM_{-\infty}^{ \left\{ s_{i_j} \right\}}$. 
%
%\subsection{Cylindrical case}
%Reasoning by the monotonicity of $\tE$ again, if $\tM_{\infty}$ is a cylinder, then $\tM_{-\infty}^{ \left\{ s_{i_j} \right\}}$ is a cylinder or a sphere. 
%However, by~\cite{huisph}, if $\tM(s)$ is ever convex, $\tM_{\infty}$ must be a sphere. 
%Of course, if $\tM_{-\infty}^{ \left\{ s_{i_j} \right\}}$ is a sphere, then there must be some $s<0$ for which $\tM(s)$ is $C^2$-close to a sphere, and is therefore convex.  
%We would have a contridiction.
%So, in this case, $\tM_{-\infty}^{ \left\{ s_{i_j} \right\}}$ must be a cylinder. 
%
%\subsection{So?}
%
%So in either case for the shape of $\tM_{\infty}$, $\tM_{-\infty}^{ \left\{ s_{i_j} \right\}}$ must be the same shape (though, at this point, possibly rotated). 
%Since all such subsequences are uniquely determined, the same is true of all sequences. 
%If we can rule out the possible rotation, then we can say $\dsp \tM_{-\infty} = \lim_{s \to -\infty} \tM(s) = \tM_{\infty}$. 
%

Finally, since $\tM_{+\infty}$ and $\tM_{-\infty}$ are isometric and both centered at 0, they have the same Gaussian area.
However, the axis of $\tM_{-\infty}$ could depend on the subsequence. 
We address this issue in the following propostion. 
\begin{prop}
	\thlabel{prop:unique}
	Let $M$ be as in \thref{prop:limit}. 
	Then $\tM_{-\infty} \equiv \tM(s) \equiv \tM_{+\infty}$. 
\end{prop}

\begin{proof}
	From \thref{prop:limit}, $\tM_{-\infty}$ and $\tM_{+\infty}$ are isometric. 
	Then we can write
	\[
		0 
		= \tE_{-\infty} - \tE_{\infty}
		= \int_{\infty}^{-\infty} \int_{\tM(s)} \left| \tF^{\perp}-\tH \tn \right|^2 \tr \dtm \: ds
	\]
	Thus we conclude that $\left( \bd_s \tF \right)^{\perp}  =\tF^{\perp}-\tH \tn= 0$ for all time. 
	This means, up to tangential diffeomorphism, that $\tM(s)$ is stationary. 
	Thus $\tM(s)$ is a fixed sphere or cylinder.

\end{proof}

\begin{proof}[Proof of Main Theorem]
	Without loss of generality, assume $M$ has a singularity at (0,0). 
	By \thref{prop:limit} and \thref{prop:unique}, $\tM(s)$ is either a stationary sphere or cylinder centered at the origin. 
	This corresponds to a homothetically shrinking $M(t)$ that is a sphere or cylinder. 
\end{proof}

	\bibliographystyle{plain}
\bibliography{bib}

\begin{thebibliography}{10}

\bibitem{and}
B~Andrews.
\newblock \textit{Non-Collapsing in Mean-Convex Mean Curvature Flow}.
\newblock {\em Geom. Topol.}, 2012.

\bibitem{dhs}
P~Daskalopoulos, R~Hamilton, and N~\v{S}e\v{s}um.
\newblock \textit{ Classification of Compact Ancient Solutions to the Curve
  Shortening Flow}.
\newblock {\em J. Differential Geom.}, 2010.

\bibitem{eck}
K.~Ecker.
\newblock \textit{Regularity Theorey for Mean Curvature Flow}.
\newblock 2004.

\bibitem{gk}
Y~Giga and R~Kohn.
\newblock \textit{Asymptotically Self-Similar Blow-up of Semilinear Heat
  Equations}.
\newblock {\em Comm. on Pure Appl. Math.}, 1985.

\bibitem{hh}
R~Haslhofer and O~Hershkovits.
\newblock \textit{Ancient Solutions of the Mean Curvature Flow}.
\newblock {\em arXiv:1308.4095}, 2013.

\bibitem{hk}
R~Haslhofer and B~Kleiner.
\newblock Mean curvature flow of mean convex hypersurfaces.
\newblock {\em Comm. Pure Appl. Math.}, 2016.

\bibitem{hui}
G.~Huisken.
\newblock \textit{Asymptotic Behavior For Singularities of the Mean Curvature
  Flow}.
\newblock {\em J. Differential Geom.}, 1990.

\bibitem{hs}
G~Huisken and C~Sinestrari.
\newblock \textit{Convex Ancient Solutions of the mean Curvature Flow}.
\newblock {\em J. Differential Geom.}, 2015.

\bibitem{mant}
C.~Mantegazza.
\newblock {\em \textit{Lecture Notes on Mean Curvature Flow}}.
\newblock 2011.

\bibitem{sw}
W~Sheng and X~Wang.
\newblock \textit{Singularity Profile in the Mean Curvature Flow}.
\newblock {\em Methods Appl. Anal.}, 2009.

\bibitem{son}
K~Sonnanburg.
\newblock \textit{Blow-up Continuity of Type-I, Mean Convex Mean Curvature
  Flow}.
\newblock {\em arXiv:1703.02619}, 2017.

\bibitem{sonc}
K~Sonnanburg.
\newblock \textit{Graph Convergence for Non-Collapsed Hypersurfaces with
  Bounded Curvature}.
\newblock {\em To Appear}, 2017.

\bibitem{whimc}
B~White.
\newblock \textit{The Nature of Singularities in Mean Curvature Flow of
  Mean-Convex Sets}.
\newblock {\em J. Amer. Math. Soc.}, 2002.

\end{thebibliography}
\end{document}